\documentclass[runningheads,a4paper]{llncs}

\usepackage{amssymb, amsmath}
\usepackage{graphicx}

\usepackage[utf8]{inputenc}
\usepackage{hyperref}
\usepackage{dsfont}
\usepackage{cancel}
\usepackage{algorithm2e}
\usepackage{bm}
\usepackage{enumitem}
\usepackage{tikz}
\usepackage{ifthen}

\usetikzlibrary{patterns}

\newcommand{\Version}{long}

\makeatletter
\renewcommand{\@biblabel}[1]{[#1]\hfill}
\makeatother

\newcommand{\indic}[1]{\mathds{1}_{\left\{ #1\right\}}}
\renewcommand{\P}{\mathds{P}}
\newcommand{\E}{\mathds{E}}
\renewcommand{\indent}{\hspace*{0.5cm}}

\newcommand{\R}{\mathbb{R}}
\renewcommand{\epsilon}{\varepsilon}
\renewcommand{\b}{\bm}

\newcommand{\diam}{\mathop{\mathrm{diam}}}
\newcommand{\node}{\mathrm{node}}
\renewcommand{\hat}{\widehat}

\usepackage{color}

\setitemize{noitemsep, topsep=0pt, parsep=0pt, partopsep=0pt, leftmargin=2em}

\newcommand{\ourtitle}{A consistent deterministic regression tree for non-parametric prediction of time series}
\title{\ourtitle}

\begin{document}
\titlerunning{A deterministic regression tree for non-parametric prediction of time series}

\author{Pierre Gaillard\inst{1}\inst{2}
\and Paul Baudin\inst{3}
}
\authorrunning{}

\institute{
EDF R\&D, Clamart, France \\
\and GREGHEC (HEC Paris, CNRS), Jouy-en-Josas, France \\
\email{pierre-p.gaillard@edf.fr}
\and Inria, Roquencourt, France \\
\email{paul.baudin@inria.fr} 
}

\maketitle

\begin{abstract}
We study online prediction of bounded stationary ergodic processes. To do so, we consider the setting of prediction of individual sequences and build a deterministic regression tree that performs asymptotically as well as the best $L$-Lipschitz constant predictors. Then, we show why the obtained regret bound entails the asymptotical optimality with respect to the class of bounded stationary ergodic processes.
\end{abstract}

\section{Introduction}

We suppose that at each time step $t=1,2,\dots$, the learner is asked to form a prediction $\hat Y_t$ of the next outcome $Y_t \in [0,1]$ of a bounded stationary ergodic process $(Y_t)_{t=-\infty,\dots,\infty}$ with knowledge of the past observations $Y_1,\dots,Y_{t-1}$. To evaluate the performance, a convex and $M$-lipschitz loss function $\ell: [0,1]^2 \to [0,1]$ is considered. The following fundamental limit has been proven by \cite{Algoet1994}. For any prediction strategy, almost surely
\begin{equation}
	\liminf_{T \to \infty} \Biggl(  \frac{1}{T}  \sum_{t=1}^T \ell(\hat Y_t, Y_t) \Biggr) \geq L^\star \,, \mbox{ where }
	L^\star = \E \! \left[ \inf_{f \in \mathcal{B}^{\infty}\!} \E \Bigl[ \ell\bigl(f(Y_{-\infty}^{-1}), Y_0 \bigr) \big| Y_{-\infty}^{-1}\Bigr]\right]
	\label{eq:Lstar} 
\end{equation}
is the expected minimal loss over all possible Borel estimations of the outcome $Y_0$ based on the infinite past ($\mathcal{B}^{\infty}$ denotes the set of Borel functions from $[0,1]^\infty$ to $[0,1]$).
One may thus try to design \emph{consistent} strategies that achieve the lower bound, that is, $\limsup_T \bigl\{ (1/T)\sum_t \ell(\hat Y_t,Y_t)\bigr\} \leq L^\star$.

\medskip
\noindent
{\bf Litterature review.} Many forecasting strategies have been designed to this purpose. The vast majority of these strategies are based on statistical techniques used for time-series prediction, going from parametric models like autoregressive models (see \cite{BrockwellDavis1991}) to non-parametric methods (see the reviews of \cite{GyoerfiHaerdleSardaVieu1989,Bosq1996,MerhavFeder1998}). 
In recent years, another collection of algorithms resolving related problems have been designed in \cite{GyoerfiLugosiFargas2001,GyorfiOttucsak2007,BiauBleakleyGyorfiOttucsak2010,BiauPatra2011}. 
At their cores, all these algorithms use some machine learning non-parametric prediction scheme (like histogram, kernel, or nearest neighbor estimation) with parameters by given both a window, and the length of the past to consider. Then, they output predictions by mixing the countably infinite set of experts corresponding to strategies with fixed values of these two parameters. 

\smallskip
\noindent
{\bf Our approach.} We adopt the point of view of individual sequences, see the monograph of \cite{Cesa-BianchiLugosi2006}. In the process, we divide into two separate layers the setting of stochastic time series and the one of individual sequences.
Our main result is Theorem~\ref{th:Lstar} and it states that any strategy that satisfies some deterministic regret bound is consistent. Section~\ref{sec:IndividualSequences} and~\ref{sec:Autoregressive} design such a strategy and consider the following framework of sequential prediction of individual sequences. We suppose that a sequence $(\b x_t,y_t) \in \mathcal{X} \times \mathcal{Y}$ is observed step by step, where $\mathcal{X} \subset [0,1]^d$ is the covariable space and $\mathcal{Y}\subset [0,1]$ a convex observation space (in Section~\ref{sec:Autoregressive}, $\b x_t$ will be replaced by $y_{t-d}^{t-1} = y_{t-d},\dots,y_{t-1}$, then, $y_t$ will be replaced by $Y_t$  in Section~\ref{sec:Lstar}).
The learner is asked at each time step $t$ to predict the next observation $y_t$ with knowledge of the past observations $y_1,\dots,y_{t-1}$ and of the past and present exogenous variables $\b x_1,\dots, \b x_{t}$.  The goal of the forecaster is to minimize its cumulative regret against the class $\mathcal{L}_L^d$ of $L$-Lipschitz functions from $[0,1]^d$ to $[0,1]$, \vspace*{-0.5em}
\[
	\hat R_{L,T} = 	\sum_{t=1}^T \ell(\hat y_t, y_t) - \inf_{f \in \mathcal{L}_L^d} \sum_{t=1}^T \ell\bigl(f (\b x_t), y_t\bigr) \,,  \vspace*{-0.5em}
\]
that is, to ensure $\hat R_{L,T} = o(T)$. 
In Section~\ref{sec:IndividualSequences}, we describe the nested EG strategy (Algorithm~\ref{alg:NestedEG}), which follows the spirit of binary regression trees like Cart (see \cite{BreimanFriedmanOlshenStone1984}). We provide in Theorem~\ref{th:regretbound} a finite-time regret bound with respect to the class of $L$-Lipschitz functions. 
We recall below the considered setting.

\smallskip \noindent
At each time step $t = 1, \dots, T$, \\
\indent 1. Forecaster observes $\b x_t \in \mathcal{X} \subset [0,1]^d$ \\
\indent 2. Forecaster predicts   $\hat y_t \in [0,1] $ \\
\indent 3. Environment chooses $y_t \in \mathcal{Y}$ \\
\indent 4. Forecaster suffers loss $\hat \ell_t = \ell(\hat y_t, y_t) \in [0,1]$.

\smallskip
\noindent
{\bf Contributions.} First,  we clean up the standard analysis of prediction of ergodic processes by carrying out the aforementioned separation in two layers.
The second advantage is the computational efficiency as we will discuss later in remarks. 
A third benefit of our approach is to be valid for a general class of loss functions when previous papers to our knowledge only treat particular cases like the square loss or the pinball loss.

\section{The nested EG strategy}
\label{sec:IndividualSequences}

The nested EG strategy (Algorithm~\ref{alg:NestedEG}) incrementally builds an estimate of the best Lipschitz function $f^{\star}$. The core idea is to estimate $f^{\star}$ precisely in areas of the covariable space $\mathcal{X}$ with many occurrences of covariables $\b x_t$, while estimating it loosely in other parts of the space. To implement this idea, Algorithm~\ref{alg:NestedEG}  maintains a deterministic binary tree whose nodes are associated with regions of the covariable space, such that the regions with nodes deeper in the tree (further away from the root) represent increasingly smaller subsets of $\mathcal{X}$ (see Figure~\ref{fig:tree}). 

In the later, we assume for simplicity that $\mathcal{X} = [0,1]^d$ and $\mathcal{Y} = [0,1]$ and that the loss function $\ell$ is from $[0,1]^2$ to $[0,1]$. 
The case of unknown bounded sets $\mathcal{X} \subset \R^d$ and $\mathcal{Y} \subset \R$ will be treated later in remarks.

\subsection{The best constant oracle}

If the number of observations such that $\b x_t$ belong to a subset $\mathcal{X}^{\node} \subset \mathcal{X}$ is small enough, one does not need to estimate $f^{\star}$ precisely over $\mathcal{X}^{\node}$.  Lemma~\ref{lem:approxlipschitz} formalizes this idea by controlling the approximation error suffered by approximating $f^{\star}$ by the best constant in $[0,1]$. The control is expressed in terms of the number of observations $T^{\node}$ and of the size of the set $\mathcal{X}^{\node}$, which is measured by its diameter defined as
	$
		\diam \!\big( \mathcal{X}^{\node} \big)= \max_{\b x, \b x' \in \mathcal{X}^{\node}} \left\| \b x - \b x'\right\|_2 \,.
	$

\begin{lemma}[Approximation of  $f^{\star}$ by a constant] Let $T^{\node} \geq 1$ and suppose that $\ell$ is $M$-Lipschitz in its first argument. Then, 
\vspace*{-0.5em}
	\[
	\inf_{y \in [0,1]} 	\sum_{t=1}^{T^{\node}} \ell(y, y_t) \leq  \inf_{f \in \mathcal{L}_L^d} \sum_{t=1}^{T^{\node}} \ell\bigl(f (\b x_t), y_t\bigr) +  M  L T^{\node} \diam\!\big(\mathcal{X}^{\node}\big),
	\]
	where $\mathcal{X}^{\node} \subset [0,1]^d$ is such that $\b x_t \in \mathcal{X}^{\node}$ for all $t=1,\dots,T^{\node}$.
	\label{lem:approxlipschitz}
\end{lemma}

\begin{proof}Let $t \geq 1$. Using that $\ell$ is $M$-Lipschitz and $f$ is L-Lipschitz, we get \vspace*{-0.3em}
\begin{equation*}
	 \ell\big(f(\b x_1), y_t\big) - \ell\bigl(f (\b x_t), y_t\bigr)
	 	 \leq M \big|f(\b x_1)-f(\b x_t)\big|
	 	 \leq M L \big\|\b x_1-\b x_t\big\|_2 \leq  M L \delta \,.
\end{equation*} \vspace*{-1.3em}

\noindent 
Summing over $t$ and noting that $\inf_y \sum_t \ell(y,y_t) \leq \sum_t \ell \big(f(\b x_1),y_t\big)$ concludes.\!\qed \end{proof}

\subsection{Performing as well as the best constant: the EG strategy}
\label{sec:EG}

\begin{algorithm}[t]
\begin{minipage}{0.9\textwidth}
\rule{\linewidth}{.5pt}
\vspace*{-0.5em}
\caption{\label{alg:EG} The gradient-based exponentially weighted average forecaster (EG) with two constant experts that predict respectively 0 and 1.}
\textbf{Parameter:} $M >0$ 

\smallskip
\textbf{For} time step $t =1,2,\dots$
	\begin{itemize}
		\item[1.] Define the learning parameter $\eta_t = M^{-1}\sqrt{(\log 2 )/t}$
		\item[2.] Predict
			\[
				\hat y_t = \frac{\exp \! \left(-\eta_t \sum_{s=1}^{t-1} \ell'(\hat y_s,y_s) \right)}{1 + \exp \! \left(-\eta_t \sum_{s=1}^{t-1} \ell'(\hat y_s,y_s) \right)} \in [0,1]\,,
			\]
			where $\ell'$ denotes the (sub)gradient of $\ell$ with respect to its first argument
		\item[3.] Observe $y_t$ 
	\end{itemize}
\rule{\linewidth}{.5pt} \vspace{.125cm}
\end{minipage}
\end{algorithm}

Lemma~\ref{lem:approxlipschitz} implies that considering constant predictions is not bad when either the covariable region is small, or the number of observations is small. The next step consists thus of estimating online the best constant prediction in~$[0,1]$.

To do so, among many existing methods, 
we consider the well-known  \emph{gradient-based exponentially weighted average forecaster} (EG), introduced by~\cite{KivinenWarmuth1997}. In the setting of prediction of individual sequences with expert advice---see the monograph by~\cite{Cesa-BianchiLugosi2006}, EG competes with the best fixed convex combination of experts. In the case where two experts predict constant predictions respectively $0$ and $1$ at all time steps, EG ensures vanishing average regret with respect to any constant prediction in $[0,1]$. We describe in Algorithm~\ref{alg:EG} this particular case of EG  and we provide the associated regret bound in Lemma~\ref{lem:EG}, whose proof follows from the standard proof of EG, available for instance in~\cite{Cesa-BianchiLugosi2006}.

\begin{lemma}[EG] 
\label{lem:EG}
Let $T^{\node} \geq 1$. We assume that the loss function $\ell$ is convex and $M$-Lipschitz in its first argument. Then, the cumulative loss of Algorithm~\ref{alg:EG} is upper bounded as follows: \vspace*{-0.3em}
\[
	\sum_{t=1}^{T^{\node}} \ell(\hat y_t,y_t) \leq \inf_{y \in [0,1]} \sum_{t=1}^{T^{\node}} \ell(y,y_t) +  2 M \sqrt{T^{\node}\log 2} \,.
\] 
\end{lemma}

\noindent
{\bf Unknown value of $M$.}
Note that Algorithm~\ref{alg:EG} needs to know in advance a uniform bound $M$ on $\ell'$. This is the case, if one considers as we do a bounded observation space $[0,1]$ with the absolute loss function, defined for all $y,y'\in [0,1]$ by $\ell(y',y) = |y-y'|$; the pinball loss, defined by $\ell_\alpha(y',y) = (\alpha - \indic{y \geq x})(y -y')$; or the square loss, defined by $\ell(y',y) = (y-y')^2$. However, in the case of an unknown observation space $\mathcal{Y}$ the bound on the gradient of the square loss is unknown and needs to be calibrated online at the small cost of the additional term $2M(2+4(\log2)/3)$ in the regret bound, see \cite{RooijErvenGrunwaldKoolen2014}.

\begin{center}
\begin{figure}[b]
\begin{center}
\begin{tikzpicture}[xscale=0.42,yscale=0.42]

\tikzstyle{arrow}=[>=latex,thick]
\tikzstyle{nodebelow}=[rectangle,draw]
\tikzstyle{recEmpty}=[fill=white]
\tikzstyle{recFilled}=[fill=gray!50!]

\def\DistanceInterLevels{3}
\def\DistanceInterLeaves{5.5}
\def\XRectangleScalingFactor{1}
\def\YRectangleScalingFactor{1}

\def\LevelA{(-0)*\DistanceInterLevels}
\def\LevelB{(-1)*\DistanceInterLevels}
\def\LevelC{(-2)*\DistanceInterLevels}
\def\InterLeaves{(1)*\DistanceInterLeaves}

\def\XWest{(-1)*\XRectangleScalingFactor}
\def\XCenter{(0)*\XRectangleScalingFactor}
\def\XEast{(1)*\XRectangleScalingFactor}
\def\YSouth{(-1)*\YRectangleScalingFactor}
\def\YCenter{(0)*\YRectangleScalingFactor}
\def\YNorth{(1)*\YRectangleScalingFactor}

\def\XLabelCoordinate{1.2}
\def\YLabelCoordinate{0}

\node[nodebelow] (R) at ({(1.5)*\InterLeaves},{\LevelA}) {};
\node (R_label) at ({(1.65)*\InterLeaves+\XLabelCoordinate},{\LevelA+\YLabelCoordinate }) {$(0,1)$};
\node (R1) at ({(1.5)*\InterLeaves+\XWest},{\LevelA+\YSouth}) {};
\node (R2) at ({(1.5)*\InterLeaves+\XEast},{\LevelA+\YNorth}) {};

\node[nodebelow] (Ra) at ({(0.5)*\InterLeaves},{\LevelB}) {};
\node (Ra_label) at ({(0.65)*\InterLeaves+\XLabelCoordinate},{\LevelB+\YLabelCoordinate }) {$(1,1)$};
\node (Ra1) at ({(0.5)*\InterLeaves+\XWest},{\LevelB+\YSouth}) {};
\node (Ra2) at ({(0.5)*\InterLeaves+\XEast},{\LevelB+\YNorth}) {};
\node (Ra3) at ({(0.5)*\InterLeaves+\XCenter},{\LevelB+\YNorth}) {};

\node[nodebelow] (Raa) at ({(0)*\InterLeaves},{\LevelC}) {};
\node (Raa_label) at ({(0.15)*\InterLeaves+\XLabelCoordinate},{\LevelC+\YLabelCoordinate }) {$(2,1)$};
\node (Raa1) at ({(0)*\InterLeaves+\XWest},{\LevelC+\YSouth}) {};
\node (Raa2) at ({(0)*\InterLeaves+\XEast},{\LevelC+\YNorth}) {};
\node (Raa3) at ({(0)*\InterLeaves+\XCenter},{\LevelC+\YNorth}) {};
\node (Raa4) at ({(0)*\InterLeaves+\XCenter},{\LevelC+\YCenter}) {};

\node[nodebelow] (Rab) at ({(1)*\InterLeaves},{\LevelC}) {};
\node (Rab_label) at ({(1.15)*\InterLeaves+\XLabelCoordinate},{\LevelC+\YLabelCoordinate }) {$(2,2)$};
\node (Rab1) at ({(1)*\InterLeaves+\XWest},{\LevelC+\YSouth}) {};
\node (Rab2) at ({(1)*\InterLeaves+\XEast},{\LevelC+\YNorth}) {};
\node (Rab3) at ({(1)*\InterLeaves+\XCenter},{\LevelC+\YNorth}) {};
\node (Rab4) at ({(1)*\InterLeaves+\XWest},{\LevelC+\YCenter}) {};

\node[nodebelow] (Rb) at ({(2.5)*\InterLeaves},{\LevelB}) {};
\node (Ra_label) at ({(2.65)*\InterLeaves+\XLabelCoordinate},{\LevelB+\YLabelCoordinate }) {$(1,2)$};
\node (Rb1) at ({(2.5)*\InterLeaves+\XWest},{\LevelB+\YSouth}) {};
\node (Rb2) at ({(2.5)*\InterLeaves+\XEast},{\LevelB+\YNorth}) {};
\node (Rb3) at ({(2.5)*\InterLeaves+\XCenter},{\LevelB+\YSouth}) {};

\node[nodebelow] (Rba) at ({(2)*\InterLeaves},{\LevelC}) {};
\node (Rba_label) at ({(2.15)*\InterLeaves+\XLabelCoordinate},{\LevelC+\YLabelCoordinate }) {$(2,3)$};
\node (Rba1) at ({(2)*\InterLeaves+\XWest},{\LevelC+\YSouth}) {};
\node (Rba2) at ({(2)*\InterLeaves+\XEast},{\LevelC+\YNorth}) {};
\node (Rba3) at ({(2)*\InterLeaves+\XCenter},{\LevelC+\YSouth}) {};
\node (Rba4) at ({(2)*\InterLeaves+\XEast},{\LevelC+\YCenter}) {};

\node[nodebelow] (Rbb) at ({(3)*\InterLeaves},{\LevelC}) {};
\node (Rbb_label) at ({(3.15)*\InterLeaves+\XLabelCoordinate},{\LevelC+\YLabelCoordinate }) {$(2,4)$};
\node (Rbb1) at ({(3)*\InterLeaves+\XWest},{\LevelC+\YSouth}) {};
\node (Rbb2) at ({(3)*\InterLeaves+\XEast},{\LevelC+\YNorth}) {};
\node (Rbb3) at ({(3)*\InterLeaves+\XCenter},{\LevelC+\YSouth}) {};
\node (Rbb4) at ({(3)*\InterLeaves+\XCenter},{\LevelC+\YCenter}) {};

\draw[arrow] (R)--(Ra);
\draw[arrow] (Ra)--(Raa);
\draw[arrow] (Ra)--(Rab);
\draw[arrow] (R)--(Rb);
\draw[arrow] (Rb)--(Rba);
\draw[arrow] (Rb)--(Rbb);

\draw[recEmpty] (R1) rectangle (R2);
\draw[recFilled] (R1) rectangle (R2);

\draw[recEmpty] (Ra1) rectangle (Ra2);
\draw[recFilled] (Ra1) rectangle (Ra3);

\draw[recEmpty] (Rb1) rectangle (Rb2);
\draw[recFilled] (Rb3) rectangle (Rb2);

\draw[recEmpty] (Raa1) rectangle (Raa2);
\draw[recEmpty] (Raa1) rectangle (Raa3);
\draw[recFilled] (Raa1) rectangle (Raa4);

\draw[recEmpty] (Rab1) rectangle (Rab2);
\draw[recEmpty] (Rab1) rectangle (Rab3);
\draw[recFilled] (Rab4) rectangle (Rab3);

\draw[recEmpty] (Rba1) rectangle (Rba2);
\draw[recEmpty] (Rba3) rectangle (Rba2);
\draw[recFilled] (Rba3) rectangle (Rba4);

\draw[recEmpty] (Rbb1) rectangle (Rbb2);
\draw[recEmpty] (Rbb3) rectangle (Rbb2);
\draw[recFilled] (Rbb4) rectangle (Rbb2);

\end{tikzpicture}
\caption{Representation of the binary tree in dimension $d=2$.}
\label{fig:tree}
\end{center}
\end{figure}
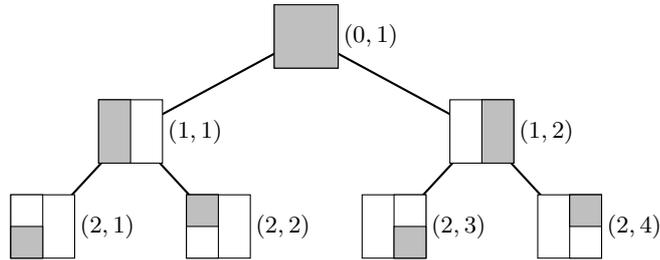
\end{center}

\subsection{The nested EG strategy}
\label{sec:NestedEG}

The nested EG strategy presented in Algorithm~\ref{alg:NestedEG} implements the idea of Lemma~\ref{lem:approxlipschitz} and Lemma~\ref{lem:EG}. It maintains a binary tree whose nodes are associated with regions of the covariable space $[0,1]^d$. The nodes in the tree are indexed by pairs of integers $(h,i)$; where the first index $h \geq 0$ denotes the distance of the node to the root (also referred to as the depth of the node) and the second index $i$ belongs to $\{1,\dots,2^h\}$. The root is thus denoted by $(0,1)$. By convention, $(h+1,2i-1)$ and $(h+1,2i)$ are used to refer to the two children of node $(h,i)$. Let $\mathcal{X}^{(h,i)}$ be the region associated with node $(h,i)$. By assumption, these regions are hyper-rectangle and must satisfy the constraints
\[
	\mathcal{X}^{(0,1)} = [0,1]^d \qquad \mbox{ and } \qquad \mathcal{X}^{(h,i)} = \mathcal{X}^{(h+1,2i-1)} \sqcup \mathcal{X}^{(h+1,2i)} \,,
\]
where $\sqcup$ denotes the disjoint union.
The set of regions associated with terminal nodes (or leaves) forms thus a partition of $[0,1]^d$.

At time step $t$, when a new covariable $\b x_t$ is observed, Algorithm~\ref{alg:NestedEG} first selects the associated leaf $(h_t,i_t)$ such that $\b x_t \in \mathcal{X}^{(h_t,i_t)}$ (step 2). The leaf $(h_t,i_t)$ then predicts the next observation $y_t$ by updating a local version $\mathcal{E}^{(h_t,i_t)}$ of Algorithm~\ref{alg:EG} (step 3). Namely, $\mathcal{E}^{(h_t,i_t)}$ runs Algorithm~\ref{alg:EG} on the sub-sequence of observations $(\b x_s, y_s)$ such that the associated leaf is $(h_t,i_t)$, that is, $(h_s,i_s) = (h_t,i_t)$. When the number of observations $T^{(h_t,i_t)}$ received and predicted by leaf $(h_t,i_t)$ becomes too large compared to the size of the region $\mathcal{X}^{(h_t,i_t)}$  (step 6), the tree is updated. To do so, the region $\mathcal{X}^{(h_t,i_t)}$ is divided in two sub-regions of equal volume by cutting along one given coordinate. 

The coordinate $r_t + 1$ to be split is chosen in a deterministic order, where $r_t = (h_t \mod d)$ and  $\mod$ denotes the modulo operation. Thus, at the root node $(0,1)$ the first coordinate is split, then by going down in the tree we split the second one, then the third one and so on until we reach the depth $d$, in which case we split the first coordinate for the second time.  
Each sub-region is associated with a child of node $(h_t,i_t)$. Consequently, $(h_t,i_t)$ becomes an inner node and is thus no longer used to form predictions. 

To facilitate the formal study of the algorithm, we will need some additional notation. In particular, we will introduce time-indexed versions of several quantities. $\mathcal{T}_t$ denotes the tree stored by Algorithm~\ref{alg:NestedEG} at the beginning of time step $t$. The initial tree is thus the root $\mathcal{T}_0 = \left\{(0,1)\right\}$ and it is expanded when the splitting condition (step 6) holds, as \vspace*{-0.3em}
\[
	\mathcal{T}_{t+1} =  \mathcal{T}_t \cup \bigl\{ (h_t+1,2i_t-1),(h_t+1,2i_t)\bigr\} \,
\]
(step 6.3) and remains unchanged otherwise. We denote by $N_t$ the number of nodes of $\mathcal{T}_t$ and by $H_t$ the height of $\mathcal{T}_t$, that is, the maximal depth of the leaves of $\mathcal{T}_t$. A performance bound for Algorithm~\ref{alg:NestedEG} is provided below.

\begin{theorem}
\label{th:regretbound}
Let $T \geq 1$ and $d \geq 1$. Then, the cumulative regret $\hat R_{L,T}$ of Algorithm~\ref{alg:NestedEG} is upper bounded as \vspace*{-0.3em}
\begin{align*}
		\sum_{t=1}^T \ell(\hat y_t, y_t) - \inf_{f \in \mathcal{L}_L^d} \sum_{t=1}^T \ell\bigl(f (\b x_t), y_t\bigr)
		& \leq   M  \left(3+ L \right) \sqrt{N_T T}  \\
		& \leq  M(3+L) \Bigl(\sqrt{T} + 2(3d)^{\frac{d}{2(d+2)}}T^{\frac{d+1}{d+2}}\Bigr) \,.
\end{align*}
\end{theorem}

\begin{algorithm}[t]
\begin{minipage}{0.9\textwidth}
\rule{\linewidth}{.5pt}
\vspace*{-0.5em}

\caption{\label{alg:NestedEG} Sequential prediction of function via Nested EG}
\textbf{Initialization:}
\begin{itemize} 
	\item[-] $\mathcal{T} = \bigl\{(0,1)\bigr\}$ a tree (for now reduced at a root node)
	\item[-] Define the bin $\mathcal{X}^{(0,1)} = [0,1]^d$
	\item[-] Start $\mathcal{E}^{(0,1)}$ a replicate of Algorithm~\ref{alg:EG}
\end{itemize}

\smallskip
\textbf{For} $t=1,\dots,T$ 
\begin{itemize}
	\item[1.] Observe $\b x_{t} \in [0,1]^d$ 
	\item[2.] Select the leaf $(h_t,i_t)$ such that $\b x_t \in \mathcal{X}^{(h_t,i_t)}$
	\item[3.] Predict according to $\mathcal{E}^{(h_t,i_t)}$ 
	\item[4.] Observe $y_t$ and feed $\mathcal{E}^{(h_t,i_t)}$ with it
	\item[5.] Update the number of observations predicted by $\mathcal{E}^{(h_t,i_t)}$ \\
	\indent \indent $T^{(h_t,i_t)} \leftarrow \# \bigl\{1\leq s\leq t , \quad (h_s,i_s) = (h_t,i_t)\bigr\}$
	
	\item[6.] \textbf{If} the splitting condition $T^{(h_t,i_t)}+1 \geq \Bigl(\diam\bigl(\mathcal{X}^{(h_t,i_t)}\bigr)\Bigr)^{-2}$ holds \textbf{then} extend the binary tree $\mathcal{T}$ as follows:
		\begin{itemize}
			\item[6.1.] Compute the decomposition $h_t = k_td+r_t$ with $r_t \in \{0,\dots,d-1\}$
			\item[6.2.] Split coordinate $r_t+1$ for node $(h_t,i_t)$
			\begin{itemize}
				\item[6.2.1.] Define the splitting threshold $\tau =  \bigl(x^- + x^+\bigr) / 2$\,, where \\
					 $x^- = \inf_{\b x \in \mathcal{X}^{(h_t,i_t)}} \{x_{r_t+1} \}$ 
					and $x^+ = \sup_{\b x \in \mathcal{X}^{(h_t,i_t)}}\{x_{r_t+1} \}$.
				\item[6.2.2.] Define two children leaves for node $(h_t,i_t)$:
				\begin{itemize}
					\item[-] the left leaf $(h_t+1,2i_t-1)$ with corresponding bin \\ 
					\hspace*{0.25cm} $ \mathcal{X}^{(h_t+1,2i_t-1)} = \{\b x \in \mathcal{X}^{(h_t,i_t)}\, : \  x_{r_t+1}  \in [x^-, \tau[ \} $
					\item[-] the right leaf $(h_t+1,2i_t)$ with corresponding bin \\ \vspace*{-0.3cm}
					 \[ \quad \mathcal{X}^{(h_t+1,2i_t-1)} = \left\{\b x \in \mathcal{X}^{(h_t,i_t)}\,:  \begin{array}{ll} x_{r_t+1} \in [\tau,x^+[  & \mbox{if } x_+ <1 \\  x_{r_t+1} \in [\tau,1] & \mbox{if } x_+ =1 \end{array}  \right\}\]
				\end{itemize}
				\item[6.2.3.] Update $\mathcal{T} \leftarrow \mathcal{T} \cup \bigl\{ (h_t+1,2i_t-1),(h_t+1,2i_t)\bigr\}$
			\end{itemize}
		\end{itemize}
	\end{itemize}
\rule{\linewidth}{.5pt} \vspace{.125cm}
\end{minipage}
\end{algorithm}

\medskip
\noindent 
{\bf Time and storage complexity.} The following lemma provides time and storage complexity guarantees for Algorithm~\ref{alg:NestedEG}. It upper bounds the maximal size of $\mathcal{T}_T$, that is, its number of nodes $N_T$ and its depth $H_T$, which yields in particular the regret bound of order $O\big(T^{(d+1)/(d+2)}\big)$ stated in Theorem~\ref{th:regretbound}. 
 \begin{lemma}
	\label{lem:boundNT}
	Let $T \geq 1$ and $d \geq 1$. Then the depth $H_T$ and the number of nodes $N_T$  of the binary tree $\mathcal{T}_T$ stored by Algorithm~\ref{alg:NestedEG} after $T$ time steps are upper bounded as follows:
	\vspace*{-0.3em}
		\[
			H_T \leq  1 + \frac{d}{2} \log_2 (4dT) \qquad \mbox{and} \qquad N_T \leq  1+ 8 \left(dT\right)^{\frac{d}{d+2}} \,.
		\]
\end{lemma}
Indeed, Algorithm~\ref{alg:NestedEG} needs to store a constant number of parameters at each node of the tree. Thus the space complexity is of order $O(N_T) = O\big(T^{d/(d+2)}\big)$. Besides at each time step $t$, Algorithm~\ref{alg:NestedEG} needs to perform $O(H_t) = O(\log t)$ binary test operations in order to select the leaf $(h_t,i_t)$. It then only needs constant time to update both $\mathcal{E}^{(h_t,i_t)}$ and $\mathcal{T}$. Thus the per-round time complexity of Algorithm~\ref{alg:NestedEG} is of order $O(\log t)$ and the global time complexity is of order $O(T \log T)$. Therefore, we can summarize:
\[
	\mbox{Storage complexity:} \quad O\big(T^{d/(d+2)}\big)\,, \qquad \mbox{Time complexity:}\quad O\big(T \log T\big) \,.
\]

 \smallskip
 \noindent
 {\bf Unknown bounded sets $\mathcal{X} \subset \R^d$~ and $\mathcal{Y} \subset \R$.} 
As we mentioned in the end of Section~\ref{sec:EG}, the generalization of Algorithm~\ref{alg:EG} and thus of Algorithm~\ref{alg:NestedEG} to an unknown set $\mathcal{Y} \subset \R$  can be obtained by using standard tools of individual sequences---see for instance \cite{RooijErvenGrunwaldKoolen2014}. To adapt Algorithm~\ref{alg:NestedEG} to any unknown compact set $\mathcal{X} \subset \R^d$, one can first divide the covariable space $\R^d$ in hyper-rectangle subregions of the form $[n_1,n_1+1] \times \dots \times [n_d,n_d+1]$ and then run independent versions of Algorithm~\ref{alg:NestedEG} on all of these subregions. If $\diam (\mathcal{X}) \leq \sqrt{d} B$ with an unknown value of $B > 0 $, then the number of initial subregions is upper-bounded by $\lceil B \rceil^d$ and by Jensen's inequality, this adaptation would lead to a multiplicative cost of  $\lceil B \rceil^{d/(d+2)}$ in the upper-bound of Theorem~\ref{th:regretbound}.

\medskip
\noindent
{\bf Comparison with other methods.} One may want to obtain similar guarantees by considering other strategies like uniform histograms, kernel regression, or nearest neighbors, which were studied in the context of stationary ergodic processes by \cite{GyoerfiLugosiFargas2001,GyorfiOttucsak2007,BiauBleakleyGyorfiOttucsak2010,BiauPatra2011}. 
We were unfortunately unable to provide any finite-time and deterministic analysis neither for kernel regression nor for nearest neighbors estimation.  
The regret bound of Theorem~\ref{th:regretbound} can however be obtained in an easier manner with uniform histograms. To do so, one can consider the class of uniform histograms $\mathcal{H}_{N}$. We divide the covariable space $[0,1]^d$ in a partition $(I_j)_{j=1,\dots,N}$ of $N$ subregions of equal size.
We define $\mathcal{H}_N$ as the class of $2^N$ prediction strategies that predict the constant values $0$ or $1$ in each bin of the partition.
Competing with this class $\mathcal{H}_N$ of $2^N$ functions by resorting for instance to EG gives the regret bound \vspace*{-0.3em}
\[
	\sum_{t=1}^T \ell\big(\hat y_t, y_t\big) \leq \min_{\b z \in [0,1]^N} \sum_{t=1}^T \ell \biggl( \sum_{j=1}^N z_j \mathds{1}_{I_j}(\b x_t), y_t \biggr) + 2 M \sqrt{T N} \,.
\]
Now, optimizing the number  $N$ of bins in hindsight (or by resorting to the doubling trick) provides a regret bound of order $O\big(T^{(d+1)/(d+2)}\big)$ against any Lipschitz function. The size of the class $\mathcal{H}_N$ is however exponential in $N = O\big(T^{d/(d+2)}\big)$, which makes the method computationally inefficient.

However, in the worst case the nested EG strategy has no better guarantee. Such worst case occurs for large number $N_T$ of nodes, which happens in particular when the trees are height-balanced, that is, when the covariables $\b x_t$ are uniformly distributed in $[0,1]^d$. But the nested EG strategy adapts better to data.
If the covariables $\b x_t$ are non-uniformly allocated (with regions of the space $[0,1]^d$ associated with much more observations than in other regions of similar size), the resulting tree $\mathcal{T}_T$ will be un-balanced, leading to a smaller number of nodes. In the best case, $N_T = O(H_T)$, which yields a regret of order $O(\sqrt{T\log T})$.
By improving the definition of Algorithm~\ref{alg:NestedEG}, one can even obtain the optimal and expected $O(\sqrt{T})$ regret if $(\b x_t)$ is constant. To do so, it only needs to compute online the effective range of the data that belongs to each node $(h,i)$, \vspace*{-0.4em}
\[
	\delta^{(h,i)}_t = \diam \left\{ \b x_s, \quad 0 \leq s \leq t \mbox{ and } (h_s,i_s) = (h,i) \right\}
\]
and substitute the diameter $\diam \mathcal{X}^{(h,i)}$ by $\delta^{(h,i)}_{t+1}$ in the splitting condition of the algorithm (step 6).

\medskip
\noindent
{\bf Proofs.} The proofs of Theorem~\ref{th:regretbound} and Lemma~\ref{lem:boundNT} are based on the following lemma, which controls the size of the regions associated with nodes located at depth $h$ in the tree $\mathcal{T}_T$.

\begin{lemma}
	\label{lem:bin}
	Let $h \geq 0$. Then, for all indices $i=1,\dots,2^h$, the diameter of the region $\mathcal{X}^{(h,i)}$ associated with node $(h,i)$ in Algorithm~\ref{alg:NestedEG} is upper bounded as \vspace*{-0.3em}
	\[p
		 \diam \! \left( \mathcal{X}^{(h,i)}\right)     \leq \sqrt{2d} 2^{-h/d} \,.
	\]	
\end{lemma}
Basically, the proof of Lemma~\ref{lem:bin} consists of an induction on the depth $h$. It is 
\ifthenelse{\equal{\Version}{long}}{postponed to Appendix~\ref{sec:proof-bin}}{deferred to an extended version of this article~\cite{GaillardBaudin2014}}.

\begin{proof}[of Lemma~\ref{lem:boundNT}]
\textbf{Upper bound for $N_T$}.
For each node $(h,i)$, we recall that $T^{(h,i)}= \sum_{t=1}^T \indic{(h_t,i_t) = (h,t)}$ denotes the number of observations predicted by using algorithm $\mathcal{E}^{(h,i)}$.  The total number of observations $T$ is the sum of $T^{(h,i)}$ over all nodes $(h,i)$. That is,
\[
	T  = \sum_{h=0}^{H_T} \sum_{i=1}^{2^h} T^{(h,i)} \indic{(h,i) \in \mathcal{T}_T}
		  \geq \sum_{h=0}^{H_T}  \sum_{i=1}^{2^h} T^{(h,i)} \ \indic{(h,i) \mbox{ is an inner node in } \mathcal{T}_T} \,.
\]
Now we use the fact that each inner node $(h,i)$ has reached its splitting condition (step 6 of Algorithm~\ref{alg:NestedEG}), that is,
$
	T^{(h,i)} + 1 \geq \bigl(\diam \bigl( \mathcal{X}^{(h,i)}\bigr)\bigr)^{-2} \,.
$
Using that $\diam\! \left( \mathcal{X}^{(h,i)}\right) \leq \sqrt{2d}2^{-h/d}$ by Lemma~\ref{lem:bin}, we get  
\vspace*{-0.3em}
\begin{align}
	T & \geq  \sum_{h=0}^{H_T}  \sum_{i=1}^{2^h} \left[-1 +  \Big(\diam \big( \mathcal{X}^{(h,i)}\big)\Big)^{-2}\right] \indic{(h,i) \mbox{ is an inner node }} \nonumber \\
		& \geq \sum_{h=0}^{H_T}    \underbrace{\left(-1+ \frac{2^{2h/d}}{2d}\right)}_{g(h)}  \underbrace{\sum_{i=1}^{2^h}  \indic{(h,i) \mbox{ is an inner node }}}_{n_h }\,. \label{eq:lowerT}
\end{align}
Because $g:\R_+ \to \R$ is convex in $h$, by Jensen's inequality
\[
	T \geq  N_T^{\mathrm{in}} \  g \!\left( \frac{1}{N_T^{\mathrm{in}}}\sum_{h=0}^{H_T} h n_h \right),
\]
where $ N_T^{\mathrm{in}} = \sum_h n_h$ is the total number of inner nodes. 
Now, by Lemma~\ifthenelse{\equal{\Version}{long}}{\ref{lem:tree} in Appendix \ref{sec:proof-tree}}{8 in the extended version of this article~\cite{GaillardBaudin2014}}, because $\mathcal{T}_T$ is a binary tree with $N_T$ nodes in total, it has exactly $N_T^{\mathrm{in}} = (N_T-1)/2$ inner nodes and the average depth of its inner nodes is lower-bounded as \vspace*{-0.5em}
\[
	\frac{1}{N_T^{\mathrm{in}}}\sum_{h=0}^{H_T} h n_h \geq \log_2 \! \left(\frac{N_T-1}{8}\right) \,.
\]
Substituting in the previous bound, it implies
\begin{align*}
	T & \geq \frac{N_T -1}{2} g\biggl( \log_2 \Bigl(\frac{N_T-1}{8}\Bigr) \biggr) = \frac{N_T-1}{2} \left( -1 + \frac{1}{2d} 2^{\frac{2}{d} \log_2 \bigl( (N_T-1)/8\bigr)} \right) \\
		& = - \frac{N_T-1}{2} + \frac{N_T-1}{4d} \left( \frac{N_T-1}{8} \right)^{2/d}  \geq  \underbrace{- \frac{N_T-1}{2}}_{\geq - T /2} + \frac{2}{d} \left( \frac{N_T-1}{8} \right)^{1+2/d} \,.
\end{align*}
By reorganizing the terms, it entails
$
	dT \geq ({3}/{4}) d T \geq \bigl( (N_T-1)/8 \bigr)^{1+2/d}
$.
Thus, $ (N_T-1)/8 \leq (dT)^{{d}/({d+2})}$, 
which yields the desired bound for $N_T$.

\medskip \noindent
\textbf{Upper bound for $H_T$}. We start from~\eqref{eq:lowerT} and we use the fact that for all $h=0,\dots,H_T-1$, there exists at least one inner node of depth $h$ in $\mathcal{T}$. Thus,
\[
	T \geq \sum_{h=0}^{H_T-1}    \left(-1+ \frac{2^{2h/d}}{2d}\right) = -H_T + \frac{1}{2d} \frac{2^{2H_T/d}-1}{2^{2/d}-1}
	\geq -H_T + \frac{2^{2(H_T-1)/d}}{2d} 
\]
where the last inequality is because $(a-1)/(b-1) \geq a/b$ for all numbers $a \geq b > 1$. 
Therefore, by upper-bounding $T \geq H_T$, we get
$
	4T \geq  {2^{2(H_T-1)/d}}/d
$
and thus $2(H_T-1)/d \leq \log_2 (4dT)$ which concludes the proof.
\qed \end{proof}

\begin{proof}[of Theorem~\ref{th:regretbound}]
The cumulative regret suffered by Algorithm~\ref{alg:NestedEG} is controlled by the sum of all cumulative regrets incurred by algorithms $\mathcal{E}^{(h,i)}$. That is,
\[	 
	\hat R_{L,T} \leq \sum_{(h,i) \in \mathcal{T}_T}  \left[ \sum_{t \in S^{(h,i)}} \ell\big(\hat y_t,y_t\big)-\inf_{f \in \mathcal{L}_L^d} \sum_{t \in S^{(h,i)}} \ell\bigl(f (\b x_t), y_t\bigr)\right] \,,
\]
where $S^{(h,i)} = \bigl\{1\leq t\leq T\,: \ (h_t,i_t) = (h,i)\bigr\}$ is the set of time steps assigned to node $(h,i)$. Now, by Lemma~\ref{lem:EG}, the cumulative loss incurred by $\mathcal{E}^{(h,i)}$ satisfies
\begin{align*}
		\sum_{t \in S^{(h,i)}} & \ell\big(\hat y_t,y_t\big)  \leq \inf_{y \in [0,1]} \sum_{t\in S^{(h,i)}} \ell\bigl(y,y_t\bigr) + {2 M \sqrt{T^{(h,i)}\log2}}  \\
	& \leq	 \inf_{f \in \mathcal{L}_L^d} \sum_{t \in S^{(h,i)}} \ell\bigl(f (\b x_t), y_t\bigr) + {M  L \hspace*{-4.5em} \underbrace{\diam \! \left(\mathcal{X}^{(h,i)}\right)}_{\leq  1/\sqrt{T^{(h,i)}} \mbox{ by step 6 of Algorithm~\ref{alg:NestedEG}} }\hspace*{-4.5em} T^{(h,i)}}   + {2 M \sqrt{T^{(h,i)}\log2}} 
\end{align*}
where the second inequality is by Lemma~\ref{lem:approxlipschitz}. Thus,
\[
	\hat R_{L,T} \leq  M  \Bigl(L + \underbrace{2\sqrt{\log 2}}_{\leq 3}\Bigr) \sum_{(h,i) \in \mathcal{T}_T}  \sqrt{T^{(h,i)}} \,.
\] \vspace*{-0.5em}

\noindent
Then, by Jensen's inequality,
\[
	\frac{1}{N_T} \sum_{(h,i)\in \mathcal{T}_T}  \sqrt{T^{(h,i)}} \leq  \sqrt{\frac{1}{N_T} \sum_{(h,i)} T^{(h,i)}} = \sqrt{\frac{T}{N_T}} \,,
\]
which concludes the first statement of the theorem. The second statement follows from Lemma~\ref{lem:boundNT} and because for all $a,b \geq 0$, $\sqrt{a+b} \leq \sqrt{a} +\sqrt{b}$,
\begin{align*}
	& M(3+L)\sqrt{N_T T}  \leq M(3+L)\sqrt{\Big(1+ 4(3dT)^{d/(d+2)}\Big)T} \\
		& \leq M(3+L) \Bigl(\sqrt{T} + \sqrt{4(3dT)^{d/(d+2)}T}\Bigr)  = M(3+L) \Bigl(\sqrt{T} + 2(3d)^{\frac{d}{2(d+2)}}T^{\frac{d+1}{d+2}}\Bigr) \,.  
\end{align*} \vspace*{-1em}
 \qed
 
\end{proof}

\section{Autoregressive framework}
\label{sec:Autoregressive}

We present in this section a technical result that will be useful for later purposes. Here, the forecaster still sequentially observes from time $t=1$ an arbitrary bounded sequence $(y_t)_{t=-\infty,\dots,+\infty}$. However, at time step $t$, it is asked to forecast the next outcome $y_t \in [0,1]$ with knowledge of the past observations $y_1^{t-1} = y_1,\dots,y_{t-1}$ only. We are interested in a strategy that performs asymptotically as well as the best model that considers the last $d$ observations to form the predictions, and this simultaneously for all values of $d \geq 1$. More formally, we denote
\[
	\hat R_{L,T}^d \triangleq \sum_{t=1}^T \ell(\hat y_t,y_t) -  \inf_{f \in \mathcal{L}_L^d} \sum_{t=1}^T \ell\big(f(y_{t-d}^{t-1}),y_t\big) \,,
\]
and we want that for all $d$, the average regrets $\hat R_{L,T}^d / T$ vanish as $T \to \infty$.
We show how it can be obtained via a meta-algorithm (Algorithm~\ref{alg:meta-algo}) that combines an increasing sequence of nested EG forecasters described in Algorithm~\ref{alg:fixed-past}. The sequence is denoted by $\mathcal{A}_1,\mathcal{A}_2, \dots$ and is such that for each $d \geq 1$,  $\mathcal{A}_d$\footnote{\label{footnote} Algorithm~$\mathcal{A}_d$ will only be used by a meta-algorithm for time steps $t \geq t_d$} forms predictions for $t \geq t_d$ for some starting time $t_d \geq 1$ and satisfies the regret bound stated in Lemma~\ref{lem:fixed-past}. 

\begin{algorithm}[htbp]
\begin{minipage}{0.9\textwidth}
\rule{\linewidth}{.5pt}
\vspace*{-0.5em}
	\caption{\label{alg:fixed-past} Forecaster $\mathcal{A}_d$ for fixed past $d$.}
	\textbf{Parameter:} $d \geq 1$ and $t_d$, a starting time
		
	\smallskip
	{\bf For} $t \leq t_{d}-1$ \\
	\indent Form no prediction$^{ \ref{footnote}}$ and observe $y_t$
	
	\smallskip
	{\bf For} $t=t_d,\dots, T$
	\begin{itemize}
		\item[1.] define $\b x_t = y_{t-d}^{t-1}$ and  feed Algorithm~\ref{alg:NestedEG} with $\b x_t \in [0,1]^d$
		\item[2.] predict $f_{d,t}$ according to Algorithm~\ref{alg:NestedEG} and feed Algorithm~\ref{alg:NestedEG} with $y_t$
	\end{itemize}
\rule{\linewidth}{.5pt} \vspace{.125cm}
\end{minipage}
\end{algorithm}

\begin{lemma}[Fixed past $\b d$]
	\label{lem:fixed-past}
	Let $T \geq 1$, $d \geq 1$, $L > 0$, and $t_d \geq d+1$. Then, Algorithm~\ref{alg:fixed-past} has a regret upper-bounded as
\[
	 \sum_{t=t_d}^T \ell(f_{d,t},y_t) -  \inf_{f \in \mathcal{L}_L^d} \sum_{t=t_d}^T \ell(f(y_{t-d}^{t-1}),y_t) \leq  M(3+L) \Bigl(\sqrt{T} + 2(3d)^{\frac{d}{2(d+2)}}T^{\frac{d+1}{d+2}}\Bigr)  \,.
\]
\end{lemma}
\begin{proof}
	The regret bound is a straightforward corollary of Theorem~\ref{th:regretbound}. \qed
\end{proof}

Now we show how to obtain the regret bound of Lemma~\ref{lem:fixed-past} simultaneously for all $d \geq 1$. To do so, we consider an increasing sequence of integers $(t_d)$ such that $t_1 = 2$. Namely, $t_d$ states at which time step algorithm $\mathcal{A}_d$ starts to form predictions and thus to be combined in Algorithm~\ref{alg:meta-algo}. 
We define at each time step $s \geq 1$ the number of active algorithms $D_s = \sup \{ d \geq 1: t_d \leq s \}$.
Basically, Algorithm~\ref{alg:meta-algo} is a meta-algorithm that combines via EG the predictions formed by all forecasters $\mathcal{A}_d$ for $d \geq 1$. Note that at time step $t$, only the $D_t$ first forecasters $\mathcal{A}_1,\dots,\mathcal{A}_{D_t}$ suggest predictions. 

%
%

\begin{algorithm}[t]
\begin{minipage}{0.9\textwidth}
\rule{\linewidth}{.5pt}
\vspace*{-0.5em}
	\caption{\label{alg:meta-algo} Extension of the Algorithm \ref{alg:NestedEG} to unknown past $d$.}
	\textbf{Parameter:} 
		\begin{itemize}
			\item $(t_d)$ an increasing sequence of starting times
			\item $(\mathcal{F}_d)_{d \geq 1}$ a sequence of forecasters such that $\mathcal{F}_d$ forms predictions for time steps $t \geq t_d$
			\item $(\eta_t)$ a sequence of learning rates
		\end{itemize}
	\smallskip
	{\bf Initialization:}  
	\begin{itemize}
		\item {\bf For} $t=1,\dots t_1-1$, predict $\hat y_t = 1/2$ 
		\item set $D_{t_1} = 1$ and $\hat{p}_{1,{t_1}} = 1$
	\end{itemize}
	
	\smallskip
	{\bf For} $t=t_1,\dots, T$
	\begin{itemize}
		\item[1.] {\bf For} each  $d = 1,\dots,D_t$, denote by $f_{d,t}$ the prediction formed by $\mathcal{F}_d$
		\item[2.] predict $\hat y_t = \sum_{d =1}^{D_t} \hat p_{d,t} f_{d,t}$		
		\item[3.] update the number of active forecasters 
			\begin{itemize}
				\item[3.1] if the next starting time occurs in $t+1$, i.e., $t_{D_{t}+1} = t+1$ then 
					\begin{itemize}
						\item[-] increase the number of forecasters by 1: \quad $D_{t+1} = D_t +1$
						\item[-] initialize the weight of the new forecaster: \quad $p_{D_{t+1},t+1} = 	1/D_{t+1}$
					\end{itemize}
				\item[3.2] otherwise if no expert starts in $t+1$, make no change: $D_{t+1} = D_t$
			\end{itemize}
		\item[4.] observe $Y_t$  and perform exponential weight update component-wise for $d=1,\dots,D_t$ as
						\[
							\hat p_{d,t+1} =
									\frac{D_t}{D_{t+1}} \  \frac{ \hat p_{d,t}^{{\eta_{t+1}}/{\eta_t}} e^{-\eta_{t+1} \ell(f_{d,t},y_t)} }{\sum_{k=1}^{D_{t}}  \hat p_{k,t}^{{\eta_{t+1}}/{\eta_t}} e^{-\eta_{t+1} \ell(f_{k,t},y_t)}  } \,.
						\]
	\end{itemize}
\rule{\linewidth}{.5pt} \vspace{.125cm}
\end{minipage}
\end{algorithm}

\medskip
Lemma~\ref{lem:meta-algo} controls the cumulative loss of Algorithm~\ref{alg:meta-algo} by the cumulative loss of the best strategy $\mathcal{F}_d$. The comparison is performed only on the time steps where $\mathcal{F}_d$ is active (i.e., forms a prediction).

\begin{lemma}
	\label{lem:meta-algo}
	Let $T \geq 1$ and $(\eta_t)_{t \geq 1}$ be a decreasing sequence of non-negative learning rates. Then, Algorithm~\ref{alg:meta-algo} satisfies for all $d \in 1,\dots, D_T \triangleq \sup\{ d, \ t_d \leq T\}$
	\vspace*{-0.6em}
	\[
		\sum_{t=t_d}^T \ell(\hat y_t,y_t)  - \ell(f_{d,t},y_t) \leq  \frac{1}{\eta_{T+1}}\log (D_{T+1} ) + \frac{1}{8} \sum_{t=t_d}^T \eta_t \,,
	\]
	which implies with learning rates $\eta_t = 2 / \sqrt{  t}$ for $t \geq 1$ the following regret bound 
	\vspace*{-0.4em}
	\[
		\sum_{t=t_d}^T \ell(\hat y_t,y_t) - \ell(f_{d,t},y_t) \leq \sqrt{T+1}  \log D_{T+1} \,.
	\]
\end{lemma}
\vspace*{-0.5em}

Note that the choice $\eta_t = \min_{s\leq t} \sqrt{ \log D_t / t}$ for $t \geq 1$ may yield the right dependency $\sqrt{\log D_T}$ in the number of experts. Similarly, the term $\sqrt{T}$ can be replaced by $\sqrt{T-t_d+1}$ by considering for instance the aggregation rule of~\cite{GaillardStoltzErven2014} with one learning rate sequence for each expert. The proof of Lemma~\ref{lem:meta-algo} follows the standard one of the exponentially weighted average forecaster. It is 
\ifthenelse{\equal{\Version}{long}}{postponed to Appendix~\ref{sec:proof-meta-algo}}{deferred to an extended version of this article~\cite{GaillardBaudin2014}}. 
It could also be recovered by noting that our setting with starting experts is almost a particular case of the setting of sleeping experts introduced in~\cite{FreundSchapireSingerWarmuth1997}. We could thus obtain similar results by following algorithms and proofs designed for this setting. We write ``almost'' because here we do not know in advance the final number of active experts, which explains the non-optimal term in $D_t$.

\begin{theorem}
\label{th:allpast}
Let $T \geq 1$, $L >0$. Let  $(t_d)$ be an increasing sequence of integers such that $t_1 = 2$.  Then, for all $d \leq D_T \triangleq \sup\{ d, \ t_d \leq T\}$, Algorithm~\ref{alg:meta-algo} run with an increasing sequence $(t_d)$ of starting times, sequence of forecasters $(\mathcal{A}_d)$ and sequence of learning rates $\eta_t = {2/\sqrt{ t}}$ satisfies
	\begin{align*}
		\hat R_{L,T}^d  & = \sum_{t=1}^T \ell(\hat y_t,y_t) -  \inf_{f \in \mathcal{L}_L^d} \sum_{t=1}^T \ell(f(y_{t-d}^{t-1}),y_t)  \\
			& \leq t_d + \sqrt{T+1} \log D_{T+1} + M(3+L) \Bigl(\sqrt{T} + 2(3d)^{\frac{d}{2(d+2)}}T^{\frac{d+1}{d+2}}\Bigr)\,.
	\end{align*}
Consequently,    for all $d \geq 1$,\  $\limsup_{T \to \infty} \left( \hat R_{L,T}^d / T \right) \leq 0$.
\end{theorem}

\begin{proof}
	The regret bound is by combining Lemma~\ref{lem:fixed-past} and Lemma~\ref{lem:meta-algo}, together with $\ell(\hat y_t,y_t) \leq 1$ for $t < t_d$. 
	 The second part is obtained by dividing by $T$ and making $T$ grows to infinity. The last part is then a consequence of Theorem~2.
\qed \end{proof}

\section{Convergence to $L^\star$}
\label{sec:Lstar}

In this section, we present our main result by deriving from Theorem~\ref{th:allpast} similar results obtained in a stochastic setting by \cite{GyoerfiLugosiFargas2001,GyorfiOttucsak2007,BiauBleakleyGyorfiOttucsak2010,BiauPatra2011}. 

We leave here the setting of individual sequences of the previous sections and we assume that the sequence of observations $y_1,\dots,y_T$ is now generated by some stationary ergodic process. More formally, we assume that a stationary bounded ergodic process $(Y_t)_{t=-\infty,\dots,\infty}$ is sequentially observed. At time step $t$, the learner is asked to form a prediction $\hat Y_t$ of the next outcome $Y_t \in [0,1]$ of the sequence with knowledge of the past observations $Y_1^{t-1} = Y_1,\dots,Y_{t-1}$. 
The nested EG strategy, as a consequence of the deterministic regret bound of Theorem~\ref{th:regretbound}, will be shown to be consistent. We recall that \cite{Algoet1994} proved that all prediction strategies verify almost surely
$
	\liminf_{T \to \infty} \bigl\{  \frac{1}{T}  \sum_{t=1}^T \ell(\hat Y_t, Y_t) \bigr\} \geq L^\star 
$, where $L^\star$, defined in \eqref{eq:Lstar}, is the expected minimal loss over all possible Borel estimations of the outcome $Y_0$ based on the infinite past. To put it another way: we cannot hope to design strategies outperforming $L^\star$. It is thus usual to require that $\sum_{t=1}^T \ell(\hat Y_t, Y_t) / T$ tends to $L^\star$ as $T \to \infty$. 

\subsection*{From individual sequences to ergodic processes}

Theorem~\ref{th:Lstar} shows that any strategy that achieves a deterministic regret bound for individual sequences as in Theorem~\ref{th:allpast} predicts asymptotically as well as the best strategy defined by a Borel function.

\medskip
Theorem~\ref{th:Lstar} will make two main assumptions on the ergodic sequence to be predicted. 
First, the sequence is supposed to lie in $[0,1]$. As earlier, this assumption can be easily relaxed to any bounded subset of $\R$---see remarks of Sections \ref{sec:EG} and \ref{sec:NestedEG}. The generalization to unbounded sequence is left to future work and should follow from \cite{GyorfiOttucsak2007}.
Second, Theorem~\ref{th:Lstar} assumes that for all $d \geq 1$ the law of $Y_{-d}^{-1}$ is regular, that is, for any Borel set $S \subset [0,1]^d$ and for any $\epsilon >0$, one can find a compact set $K$ and an open set $V$ such that
\[
	K \subset S \subset V, \qquad \mbox{and} \qquad \P_{Y_{-d}^{-1}}(V \backslash K) \leq \epsilon \,.
\]
This second assumption is considerably weaker than the assumptions required by  \cite{BiauPatra2011}  on the law of $(Y_{-d}^{-1})$ obtained for quantile prediction. The authors indeed imposed that the random variables $\| Y_{-d}^{-1} -s \|$ have continuous distribution functions for all $s \in \R^d$ and the conditional distribution function $F_{Y_0 | Y_{-\infty}^{-1}}$ to be increasing. One can however argue that their assumptions are thus hardly comparable with ours because they consider unbounded ergodic processes. We aim at obtaining in the future minimal assumptions for any generic convex loss function $\ell$ in the case of unbounded ergodic process, see~\cite{MorvaiWeiss2011}.

\begin{theorem}
\label{th:Lstar}
Let $(Y_t)_{t=-\infty,\dots,\infty}$ be a stationary bounded ergodic process. We assume that for all $t$, $Y_t \in [0,1]$ almost surely and that for all $d \geq 1$ the law of $Y_{-d}^{-1} = (Y_{-d}, \dots, Y_{-1})$ is regular. Let $\ell: [0,1]^2 \to [0,1]$ be a loss function $M$-Lipschitz in its first argument. Assume that a prediction strategy satisfies for all $d \geq 1$, \vspace*{-0.4cm}
\begin{equation*}
	\forall L\geq 0 \qquad \limsup_{T \to \infty} \left( \frac{1}{T} \sum_{t=1}^T \ell\Big(\hat Y_t,Y_t \Big) \right)  \leq  \limsup_{T \to \infty} \left(  \inf_{f \in \mathcal{L}_{L}^d} \frac{1}{T} \sum_{t=1}^T \ell\Bigl(f(Y_{t-d}^{t-1}),Y_t\Bigr) \right) \,,
\end{equation*}
then, almost surely, \vspace*{-0.4cm}
\[
	\limsup_{T \to \infty} \left( \frac{1}{T} \sum_{t=1}^T \ell \Big(\hat Y_t,Y_t\Big) \right)  \leq  L^{\star}\,.
\]
\end{theorem}

By Theorem~\ref{th:allpast}, Algorithm~\ref{alg:meta-algo} satisfies the assumption of Theorem~\ref{th:Lstar}. Our deterministic strategy is thus asymptotically optimal for any stationary bounded ergodic process satisfying the assumptions of Theorem~\ref{th:Lstar}. Here we only give the main ideas in the proof of Theorem~\ref{th:Lstar}. The complete argument is given in 
\ifthenelse{\equal{\Version}{long}}{Appendix~\ref{sec:proof-Lstar}}{an extended version of this article~\cite{GaillardBaudin2014}}. 

\begin{proof}[sketch for Theorem~\ref{th:Lstar}]
The proof follows from the one of Theorem~1 in \cite{GyoerfiLugosiFargas2001}. The new ingredient of our proof is mainly Lemma~\ref{lem:approxBorel}, which states that the best constant Lipschitz strategy performs as well as the best constant Borel strategy. First, because of Breiman's generalized ergodic theorem (see \cite{Breiman1957}) the right-term converges, and by making $L \to \infty$, we get \vspace*{-0.5em}
\[
	 \limsup_{T \to \infty} \left( \frac{1}{T}  \sum_{t=1}^T \ell \bigl( \hat Y_t,Y_t \bigr) \right)\leq \inf_{f \in \mathcal{L}^d}  \E \! \left[ \ell\bigl(f(Y_{-d}^{-1}),Y_0\bigr) \right] \,,
\]
where $\mathcal{L}^d$ is the set of Lipschitz functions from $\R^d$ to  $\R$. Then, by Lemma~\ref{lem:approxBorel} the infimum over all Lipschitz functions equals the infimum over the set $\mathcal{B}^d$ of Borel functions. Therefore, by exhibiting a well-chosen Borel function (see \cite[Theorem~8]{Algoet1994}), this yields
\begin{align*}
\limsup_{T \to \infty} \frac{1}{T}  \sum_{t=1}^T \ell \bigl( \hat Y_t,Y_t \bigr) 
	& \leq \inf_{f \in \mathcal{B}^d}  \E \Bigl[ \ell\bigl(f(Y_{-d}^{-1}),Y_0\bigr) \Bigr]  \\
	& = \E \biggl[\inf_{f \in \mathcal{B}^d}  \E \Bigl[ \ell\bigl(f(Y_{-d}^{-1} ),Y_0\bigr) \Big| Y_{-d}^{-1} \Bigr] \biggr].
\end{align*}
The proof is then concluded by making $d\to \infty$ thanks to the martingale convergence theorem.
\qed \end{proof}

\begin{lemma} 
\label{lem:approxBorel}
Let $\mathcal{X}$ be a convex and compact subset of a normed space. Let $\ell: [0,1]^2 \to [0,1]$ be a loss function $M$-Lipschitz in its first argument. Let $X$ be a random variable on $\mathcal{X}$ with a regular law $\mathbb{P}_{X}$  and let $Y$ be a random variable on $[0,1]$. Then,
\[
	\inf_{f \in \mathcal{L}^{\mathcal{X}}}  \E \! \left[ \ell\bigl(f(X),Y\bigr) \right]   = \inf_{f \in \mathcal{B}^{\mathcal{X}}}  \E \! \left[ \ell\bigl(f(X),Y\bigr) \right]  \,,
\]
where $\mathcal{L}^{\mathcal{X}}$ denotes the set of Lipschitz functions from $\mathcal{X}$ to $\R$ and $\mathcal{B}^{\mathcal{X}}$ the one of Borel functions from $\mathcal{X}$ to $\R$.
\end{lemma}

The proof of Lemma~\ref{lem:approxBorel} 
\ifthenelse{\equal{\Version}{long}}{postponed to Appendix~\ref{sec:proof-approxBorel}}{deferred to an extended version of this article~\cite{GaillardBaudin2014}} as well. It follows from the Stone-Weierstrass theorem, used to approximate continuous functions, and from Lusin's theorem, to approximate Borel functions.

\medskip \noindent
{\bf Computational efficiency.} The space complexity of Algorithm~\ref{alg:meta-algo} depends on the chosen sequence of starting times $(t_d)$. It can be arbitrary close to the space complexity of the nested EG strategy, which is $O\big(T^{d/(d+2)}\big)$. Previous algorithms of \cite{GyoerfiLugosiFargas2001,GyorfiOttucsak2007,BiauBleakleyGyorfiOttucsak2010,BiauPatra2011} exhibit consistent strategies as well. However, in practice, these algorithms involve choices of parameters somewhere in their design (by choosing the a priori weight of the infinite set of experts).  Then, the consideration of an infinite set of experts makes the exact algorithm computationally inefficient. For practical purpose, it needs to be approximated. This can be obtained by MCMC or for instance by restricting the set of experts to some finite subset at the cost, however, of loosing theoretical guarantees, see \cite{BiauPatra2011}. 

\medskip \noindent
{\bf Generic loss function.} Theorem~\ref{th:Lstar} assumes $\ell$ to be  bounded, convex, and $M$-Lipschitz in its first argument. In contrast, the results of \cite{GyoerfiLugosiFargas2001,GyorfiOttucsak2007,BiauBleakleyGyorfiOttucsak2010} only hold for the square loss (while  \cite{BiauPatra2011} extend them to the pinball-loss).

\begingroup
\let\clearpage\relax
\bibliographystyle{alpha}
\bibliography{incertitudes.bib}
\endgroup

\ifthenelse{\equal{\Version}{long}}{\newpage \appendix 
\begin{center}
{\Large
Additional Material for\\
\vspace{0.2\baselineskip}
``\ourtitle''}
\end{center}

\bigskip
\noindent
We gather in this appendix the proofs, which were
omitted from the main body of the paper

\section{Proof of Lemma~\ref{lem:bin}}
\label{sec:proof-bin}

It suffices to prove  that for all $h \geq 0$, for all indexes $i \in \{1,\dots,2^h\}$ and all coordinates $j \in \{1,\dots,d\}$, the ranges $\delta_j^{(h,i)} \triangleq \max_{\b x, \b x'\in \mathcal{X}^{(h,i)}} \bigl| x_j - x_j'\bigr|$ satisfies
		\begin{equation}
			\label{eq:induction}
			\delta_j^{(h,i)}   =  
				\left\{ \begin{array}{ll} 2^{-(k+1)} & \mbox{if } j \leq r \\ 2^{-k} & \mbox{otherwise} \end{array}\right. \,, 
		\end{equation}
	where $h = kd +r$ is the decomposition with $r \in \{0,\dots,d-1\}$.  Indeed, we then have
	\[
		\diam \! \left( \mathcal{X}^{(h,i)} \right) = \max_{\b x, \b x'\in \mathcal{X}^{(h,i)}}  {\left\|\b x-\b x'\right\|}_2 
			\leq \sqrt{\sum_{j=1}^d \Bigl(\delta_j^{(h,i)} \Bigr)^2} \,.
	\]
But by~\eqref{eq:induction}, for $r$ coordinates $j \in \{1,\dots,r\}$ among the $d$ coordinates $\delta_j^{(h,i)}$ equals $2^{-(k+1)}$ while the $d-r$ remaining coordinates $j \in \{r+1,\dots,d\}$ satisfy $\delta_j^{(h,i)} = 2^{-k}$. Thus, by routine calculations
\begin{align*}
	\diam\! \left( \mathcal{X}^{(h,i)}\right)
	& \leq \sqrt{ r \left( 2^{-(k+1)}\right)^2 + (d-r) \left( 2^{-k} \right)^2}  \\
	& = 2^{-k} \sqrt{\frac{r}{4} + d-r}  \\
	& = \sqrt{d} 2^{-k} \sqrt{1 - \frac{3r}{4d} }  \\
	& = 	 \sqrt{d} \left(2^{1/d}\right)^{-(dk+r)}  2^{r/d} \sqrt{1 - \frac{3r }{4d}}  
\end{align*}
But, 
\[
	 2^{r/d} \sqrt{1 - \frac{3r }{4d}} \leq \max_{0 \leq u \leq 1} \left\{ 2^{u} \sqrt{1 - \frac{3u }{4}} \right\}   \approx 1.12 \leq \sqrt{2} \,.
\]
The proof is concluded by substituting in the previous bound.
	
\medskip
Now, we prove~\eqref{eq:induction} by induction on the depth $h$. This is true for $h=0$ as the bin of the root node $\mathcal{X}^{(0,1)}$ equals $[0,1]^d$ by definition. 
Besides, let $h\geq 0$ and $i \in \{1,\dots,2^h \}$. We compute the decomposition $h=kd+r$ with $r \in \{0,\dots,d-1\}$. We have by step 5.4 of Algorithm~\ref{alg:NestedEG} that the range of each coordinate $j \neq r+1$ of the bin of the child node $(h+1,2i)$ remains the same 
\begin{equation} 
	\label{eq:range1}
	 \delta_j^{(h+1,2i)} = \delta_j^{(h,i)} =  
				\left\{ \begin{array}{ll} 2^{-(k+1)} & \mbox{if } j \leq r \\ 2^{-k} & \mbox{if } j \geq r+2 \end{array}\right. \,, 
\end{equation}
and the range of coordinate $r+1$ is divided by 2,
\begin{equation}
	\label{eq:range2}
	 \delta_{r+1}^{(h+1,2i)} = \delta_{r+1}^{(h,i)}/2 =  2^{-(k+1)} \,.
\end{equation}
Equations~\eqref{eq:range1} and~\eqref{eq:range2} are also true for the second child $(h+1,2i-1)$, and this concludes the induction.

\section{Lemma~\ref{lem:tree} and its proof}
\label{sec:proof-tree}

\begin{lemma}
	\label{lem:tree}
	Let $N \geq 1$ be an odd integer. Let $\mathcal{T}$ be a binary tree with $N$ nodes. Then,
	\begin{itemize}
		\item its number of inner-nodes equals $N^{\mathrm{in}} = (N-1)/2$. 
		\item the average depth (i.e., distance to the root) of its inner nodes is lower-bounded as
			\[
				\frac{1}{N^{\mathrm{in}}} \sum_{h=0}^{\infty} h \  \# \{\mbox{inner nodes in $\mathcal{T}$ of depth }h \} \geq  \log_2 \! \left( \frac{N-1}{8} \right)\,.
			\]
	\end{itemize}
\end{lemma}

\begin{proof}
{\emph{First statement}.} We proceed by induction. If $N = 1$, there is only one binary tree with one node, the lone leaf,  so that $N^{\mathrm{in}} = 0$. Now, if $\mathcal{T}$ is a binary tree with $N\geq 3$ nodes, select an inner node $n$ which is parent of two leaf nodes. Then,  replaces the subtree rooted at $n$ by a leaf node. The resulting subtree $\mathcal{T}'$ of $\mathcal{T}$ has $N-2$ nodes, so that by induction hypothesis $\mathcal{T}'$ has $(N-3)/2 $ inner nodes. But, $\mathcal{T}'$ has also $N^{\mathrm{in}} -1$ inner nodes. Therefore $N^{\mathrm{in}}  = (N-1)/2$. 

\medskip \noindent
{\emph{Second statement}.} We note that the average depth is minimized for the equilibrated binary trees, that are such that 
\begin{itemize}
	\item all depths $h \in \{0, \dots, \lfloor \log_2 N^{\mathrm{in}} \rfloor\}$ have exactly $2^h$ inner nodes;
	\item no inner nodes has depth $h > \lceil \log_2 N^{\mathrm{in}} \rceil$.
\end{itemize}
Therefore,
\[
	 \frac{1}{N^{\mathrm{in}}} \sum_{h=0}^{\infty} h \  \# \{\mbox{inner nodes in $\mathcal{T}$ of depth }h \} \geq \frac{1}{N^{\mathrm{in}}} \sum_{h=0}^{ \lfloor log_2 N^{\mathrm{in}} \rfloor} h 2^h
\]
Now, we use that $\sum_{i=0}^{n-1} i2^i = 2^{n}(n-2) +2$ for all $n \geq 1$, which implies because $\lfloor \log_2 N^{\mathrm{in}} \rfloor \geq \log_2 N^{\mathrm{in}}  - 1$ and by substituting in the previous bound,
\[
	\frac{1}{N^{\mathrm{in}}} \sum_{h=0}^{\infty} h \  \# \{\mbox{inner nodes in $\mathcal{T}$ of depth }h \} \geq  \underbrace{\frac{2^{ \log_2 N^{\mathrm{in}} }}{N^{\mathrm{in}}}}_{=1} \left( \log_2 N^{\mathrm{in}} -2 \right) + \underbrace{\frac{2}{N^{\mathrm{in}}} }_{\geq 0}\,.
\]
This concludes the proof by substituting $N^{\mathrm{in}} = (N-1)/2$.
\qed \end{proof}

\section{Proof of Lemma~\ref{lem:meta-algo}}
\label{sec:proof-meta-algo}
The proof follows from a simple adaptation of the proof of the regret bound of the exponentially weighted average forecaster---see for instance \cite{Cesa-BianchiLugosi2006}. By convexity of $\ell$ and by Hoeffding's inequality, we have at each time step $t$
\begin{align*}
	\ell(\hat y_t,y_t) \leq \sum_{d=1}^{D_t} \hat p_{d,t} \ell(f_{d,t},y_t) 
	& \leq -\frac{1}{\eta_{t}} \log \sum_{d=1}^{D_t} \hat p_{d,t} e^{-\eta_{t} \ell(f_{d,t},y_t)} + \frac{\eta_{t}}{8} 
\end{align*}
By Jensen's inequality, since $\eta_{t+1} \leq \eta_{t}$ and thus $x \mapsto x^{\eta_{t}/\eta_{t+1}}$ is convex
\begin{align*}
	\frac{1}{D_T} \sum_{d=1}^{D_t} \hat p_{d,t} e^{-\eta_{t} \ell(f_{d,t},y_t)} 
	& = \frac{1}{D_T} \sum_{d=1}^{D_t} \left(\hat p_{d,t}^{\frac{\eta_{t+1}}{\eta_t}} e^{-\eta_{t+1} \ell(f_{d,t},y_t)} \right)^{\frac{\eta_t}{\eta_{t+1}}} \\
	& \geq \left(\frac{1}{D_t} \sum_{d=1}^{D_t} \hat p_{d,t}^{\frac{\eta_{t+1}}{\eta_t}} e^{-\eta_{t+1} \ell(f_{d,t},y_t)} \right)^{\frac{\eta_t}{\eta_{t+1}}}
\end{align*}
Substituting in Hoeffding's bound we get
\begin{align*}
	\ell(\hat y_t,y_t) & \leq \left(\frac{1}{\eta_{t+1}}-\frac{1}{\eta_{t}}\right) \log {D_{t}} -\frac{1}{\eta_{t+1}} \log \! \left(  \sum_{d=1}^{D_t} \hat p_{d,t}^{\frac{\eta_{t+1}}{\eta_t}} e^{-\eta_{t+1} \ell(f_{d,t},y_t)} \right) + \frac{\eta_t}{8} 
\end{align*}
Now, by definition of the loss update in step 3 of Algorithm~\ref{alg:meta-algo}, for all $d = 1,\dots, D_t$
\begin{align*}
	\sum_{k=1}^{D_t} \hat p_{k,t}^{\frac{\eta_{t+1}}{\eta_t}} e^{-\eta_{t+1} \ell(f_{k,t},y_t)} = \frac{D_{t}}{D_{t+1}} \ \frac{\hat p_{d,t}^{\frac{\eta_{t+1}}{\eta_t}} e^{-\eta_{t+1} \ell(f_{d,t},y_t)} }{\hat p_{d,t+1}} 
\end{align*}
which after substitution in the previous bound leads to the inequality
\begin{align*}
	\ell(\hat y_t,y_t)  & \leq  \ell(f_{d,t},y_t) +  \frac{1}{\eta_{t+1}}\log (D_{t+1} \hat p_{d,t+1})  - \frac{1}{\eta_t} \log (D_t \hat p_{d,t}) + \frac{\eta_t}{8} \,.
\end{align*}
By summing over $t=t_d,\dots,T$, the sum telescopes; using that $\hat p_{d,t_d} = 1/D_{t_d}$ by step 3.1.
\[
\sum_{t=t_d}^T \ell(\hat y_t,y_t)  \leq   \sum_{t=t_d}^T \ell(f_{d,t},y_t) +  \frac{1}{\eta_{T+1}}\log (D_{T+1}  \underbrace{\hat p_{d,T+1}}_{\leq 1} )  - \frac{1}{\eta_t} \log ( \underbrace{D_{t_d} \hat p_{d,t_d}}_{= 1}) + \frac{1}{8} \sum_{t=t_d}^T \eta_t \,,
\]
which concludes the proof of the first statement. The second statement of the theorem is because
 \[
 	\frac{1}{2} \sum_{t=1}^T \eta_t = \sum_{t=1}^T \frac{1}{\sqrt{t}} = 1 + \sum_{t=2}^T \frac{1}{\sqrt{t}} \leq 1 + \int_{1}^T \frac{1}{\sqrt{t}} dt \leq  2 \sqrt{T} \,.
 \]

\section{Proof of Lemma~\ref{lem:approxBorel}}
\label{sec:proof-approxBorel}

The proof is performed in two steps. 

\bigskip \noindent
{\bf{Step 1: Lipschitz $\rightarrow$ Continuous}.} First, the Stone-Weierstrass theorem entails that any continuous function $f:\mathcal{X} \to \R$ from a compact metric space $\mathcal{X}$ to $\R$ is the uniform limit of Lipschitz functions, see e.g., ~\cite{Georganopoulos1967}. Thus, the dominated convergence theorem yields 
\[
	\inf_{f \in \mathcal{L}}  \E \Bigl[ \ell\bigl(f(X),Y\bigr) \Bigr]  = \inf_{f \in \mathcal{C}}  \E \Bigl[ \ell\bigl(f(X),Y\bigr) \Bigr] \,,
        \]
where $\mathcal{L}$ denotes the set of Lipschitz functions from $\mathcal{X}$ to $\R$ and $\mathcal{C}$ is the set of continuous functions from $\mathcal{X}$ to $\R$.

\bigskip \noindent
{\bf{Step 2: Continuous $\rightarrow$ Borel.}} Second, by the version of Lusin's theorem stated in Theorem~\ref{prop:lusin}, we can approximate any mesurable function by continuous functions (this is where regularity is used). 

Let $\delta, \epsilon > 0$ and $f: \mathcal{X} \to [0,1]$ be a Borel function. By Theorem~\ref{prop:lusin}, there exists a continuous function $g: \mathcal{X} \to [0,1]$ such that
\[
	\P_{X} \Bigl\{ |f - g| \geq \delta \Bigr\} \leq \epsilon \,.
\]
Then by Jensen's inequality, and since 
\begin{align*}
 	\Delta  & \triangleq \bigg|  \E \Bigl[ \ell\bigl(f(X),Y\bigr)\Big]  -  \E \Bigl[ \ell\bigl(g(X),Y\bigr)\Big]  \biggr| \leq \E \biggl[ \Big| \ell\bigl(f(X),Y\bigr) - \ell\bigl(g(X),Y\bigr) \Big| \biggr] \\
		& \leq \underbrace{\P_{X} \Bigl\{ |f - g| \geq \delta \Bigr\}}_{\leq \epsilon} \ +\  
		\underbrace{\E \Big[ M \big| f(X) - g(X) \big| \ \indic{| f(X) - g(X) | \leq \delta } \Big]}_{\leq M \delta}  \,,
\end{align*}
where the second inequality is because $\ell$ takes values in $[0,1]$ and is $M$-Lipschitz in its first argument. Thus $\Delta \leq \epsilon + M \delta$, which concludes the proof since this is true for arbitrary small values of $\epsilon$ and $\delta$.

\begin{theorem}[Lusin]
	\label{prop:lusin}
	 If $\mathcal{X}$ is a convex and compact subset of a normed space, equipped with  a regular probability mesure $\mu$, then for every measurable function $f: \mathcal{X} \to [0,1]$ and for every $\delta, \epsilon > 0$, there exists a continuous function $g: \mathcal{X} \to [0,1]$ such that
	\[
		\mu \left\{ \bigl| f - g \bigr| \geq \delta \right\} \leq \epsilon \,.
	\]
\end{theorem}

The proof of Theorem~\ref{prop:lusin} can be easily derived from the proof of \cite[Proposition 25]{StoltzLugosi2007}. 

\section{Proof of Theorem~\ref{th:Lstar}}
\label{sec:proof-Lstar}
In this proof, apart from the use of Breiman's generalized ergodic  theorem in the beginning and the martingale convergence theorem in the end (as exhibited in \cite{GyoerfiLugosiFargas2001,GyorfiOttucsak2007,BiauBleakleyGyorfiOttucsak2010,BiauPatra2011}), we resort to new arguments.

Let $d \geq 1$ and $L \geq 0$. Then, by assumption and by exchanging $\limsup$ and $\inf,$
\begin{align*}
	\limsup_{T \to \infty} \frac{1}{T}  \left( \sum_{t=1}^T \ell \bigl( \hat Y_t,Y_t \bigr) \right)
		& \leq  \inf_{f \in \mathcal{L}_{L}^d}\    \limsup_{T \to \infty}  \left( \frac{1}{T} \sum_{t=1}^T \ell\bigl(f(Y_{t-d}^{t-1}),Y_t\bigr) \right) \,.
\end{align*}
Because $\ell$ is bounded over $[0,1]^2$ and thus integrable, Breiman's generalized ergodic theorem (see \cite{Breiman1957}) entails that the right-term converges: almost surely,
\[
  \lim_{T \to \infty}  \left( \frac{1}{T} \sum_{t=1}^T \ell\bigl(f(Y_{t-d}^{t-1}),Y_t\bigr) \right) =  \E \! \left[ \ell\bigl(f(Y_{-d}^{-1}),Y_0\bigr) \right] 
\]
and thus,
\[
	 \limsup_{T \to \infty}  \left( \frac{1}{T}  \sum_{t=1}^T \ell \bigl( \hat Y_t,Y_t \bigr)  \right)
		 \leq \inf_{f \in \mathcal{L}_{L}^d}\  \E\Bigl[ \ell\bigl(f(Y_{-d}^{-1}),Y_0\bigr) \Bigr]  \,.
\]
By letting $L \to \infty$ in the inequality above, we get
\[
	\limsup_{T \to \infty} \left( \frac{1}{T}  \sum_{t=1}^T \ell \bigl( \hat Y_t,Y_t \bigr)  \right)
		 \leq \inf_{f \in \mathcal{L}^d}  \E \Bigl[ \ell\bigl(f(Y_{-d}^{-1}),Y_0\bigr) \Bigr]  \,.
\]
By Lemma~\ref{lem:approxBorel} the infimum over all continuous functions equals the infimum over the set $\mathcal{B}^d$ of Borel functions. Therefore,
\begin{align*}
\limsup_{T \to \infty} \left(  \frac{1}{T}  \sum_{t=1}^T \ell \bigl( \hat Y_t,Y_t \bigr) \right) 
	& \leq \inf_{f \in \mathcal{B}^d}  \E \Bigl[ \ell\bigl(f(Y_{-d}^{-1}),Y_0\bigr) \Bigr] \\
	& \leq \E \biggl[  \underbrace{\inf_{f \in \mathcal{B}^d}  \E \Bigl[ \ell\bigl(f(Y_{-d}^{-1}),Y_0\bigr) \Big| Y_{-d}^{-1} \Bigr]}_{ \triangleq Z_d} \biggr] \,,
\end{align*}
where the second inequality is by the measurable selection theorem---see Theorem~8 in Appendix I of \cite{Algoet1994}. Now, we remark that $\big(Z_d\big)$ is a bounded super-martingale with respect to the family of sigma algebras $\big( \sigma (Y_{-d}^{-1} ) \big)_{d \geq 1}$. Indeed, the function $\inf_{f \in \mathcal{B}^{d+1}}(.)$ is concave, thus conditional Jensen's inequality
\begin{align*}
	\E \bigl[Z_{d+1} \big| Y_{-d}^{-1} \bigr] 
		& \leq \inf_{f \in \mathcal{B}^{d+1}}  \E \biggl[ \E \Bigl[ \ell\Bigl(f\big(Y_{-(d+1)}^{-1}\big),Y_0\Bigr) \Big| Y_{-(d+1)}^{-1} \Bigr] \bigg| Y_{-d}^{-1} \biggr] \\
		& = \inf_{f \in \mathcal{B}^{d+1}} \E \Bigl[ \ell\Bigl(f\big(Y_{-(d+1)}^{-1}\big),Y_0\Bigr) \bigg| Y_{-d}^{-1} \Bigr]
\end{align*}
Now, we note that 
\[	
\inf_{f \in \mathcal{B}^{d+1}}  \E \Bigl[ \ell\Bigl(f\big(Y_{-(d+1)}^{-1}\big),Y_0\Bigr) \Big| Y_{-d}^{-1} \Bigr] 
	\leq \inf_{f' \in \mathcal{B}^{d}}  \E \Bigl[ \ell\Bigl(f'\big(Y_{-d}^{-1}\big),Y_0\Bigr) \Big| Y_{-d}^{-1} \Bigr] = Z_d \,,
\]
which yields $\E \bigl[Z_{d+1} \big| Y_{-d}^{-1} \bigr]  \leq Z_d$. Thus, the martingale convergence theorem (see e.g. \cite{Chow1965}) implies that $Z_d$ converges almost surely and in $\mathbb{L}_1$. Thus,
\[
	\lim_{d \to  \infty} \E \bigl[ Z_d \bigr] =  \E \! \left[  \inf_{f \in \mathcal{B}^{\infty}}  \E \Bigl[ \ell\bigl(f(Y_{-\infty}^{-1}),Y_0\bigr) \Big| Y_{-\infty}^{-1} \Bigr]\right] = L^\star \,,
\]
which yields the stated result $\limsup_T \sum_{t=1}^T \ell \bigl( \hat Y_t,Y_t \bigr) / T = L^\star$.
}{}

\end{document}